\documentclass{amsart}
\usepackage{color}
\usepackage{amssymb}

\newtheorem{theorem}{Theorem}[section]
\newtheorem{lemma}{Lemma}[section]
\newtheorem{proposition}{Proposition}[section]

\theoremstyle{definition}

\theoremstyle{remark}
\newtheorem{remark}[theorem]{Remark}

\numberwithin{equation}{section}

\renewenvironment{proof}{\medskip\noindent{\sc Proof:}}{\medskip}

\def\qed{\ifhmode\unskip\nobreak\fi\quad
  \ifmmode\square\else$\square$\fi}

\newcommand{\abs}[1]{\lvert#1\rvert}
\newcommand{\R}{\mathbb{R}}
\newcommand{\C}{\mathbb{C}}
\newcommand{\Z}{\mathbb{Z}}
\newcommand{\T}{\mathbb{T}}
\newcommand{\ov}{\overline}
\newcommand{\begeq}{\begin{equation}}
\newcommand{\stopeq}{\end{equation}}
\newcommand{\ep}{\epsilon}

\newcommand{\dis}{\displaystyle}
\newcommand{\pa}{\partial}
\newcommand{\ei}[1]{\mathrm{e}^{#1}}
\newcommand{\dd}[2]{\dis\frac{\pa #1}{\pa #2}}

\newcommand{\si}{\sigma}
\newcommand{\vfl}{\mathbf{L}}
\newcommand\supp{\text{\rm supp\,}}

\keywords{  complex vector fields, first integrals, theta function, global solvability, similarity principle }
\thanks{The second author was supported by FAPESP (grant 2017/00848-0). }

\begin{document}

\date{ \today  \ (version's date)}

\title{Class of hypocomplex structures on the two dimensional torus}

\author{Abdelhamid Meziani}
\address{Florida International University}
\email{meziani@fiu.edu}

\author {Giuliano Zugliani}
\address{Universidade Federal de S\~ao Carlos}
\email{giuzu@dm.ufscar.br}

\begin{abstract}
We study the H\"{o}lder solvability of a class of complex vector fields on the torus $\mathbb{T}^2$. We make use of the Theta function to associate a Cauchy-Pompeiu type integral operator. A similarity principle for the solutions of the equation $Lu=au+b\bar{u}$ is obtained.
\end{abstract}

\maketitle
\section{Introduction}
This paper deals with the solvability of a class of complex vector fields on the two-dimensional torus
$\T^2$. The main results are generalizations of those contained in the recent papers \cite{CDM} and \cite{CDM1}
where the focus was on solvability in domains of the plane $\R^2$.
Study of complex vector fields in $\T^2$, or on compact manifolds, was considered in many papers (see for instance \cite{BDG}, \cite{BCM}, \cite{BPZZ}, \cite{DM}, \cite{HZ}, \cite{Mez3})
under the
assumption of separation of variables: the coefficients of the induced equations are independent on certain variables. This allows
the use of partial Fourier series to carry out the analysis.
Our approach here is different and there is no need for the assumption of separation of variables and the structures are
not amenable to the use of Fourier series.

For the class of closed and locally solvable one forms $\omega=adx+bdy$ on $\T^2$ with orthogonal vector field
\[
L=b\dd{}{x}-a\dd{}{y}\, ,
\]
we associate a first integral in $Z$ on universal covering space $\R^2$. This function turns out to be
a global homeomorphism $Z:\, \R^2\longrightarrow\C$, sending a fundamental rectangle $R$ of the covering space $\R^2$ onto
a parallelogram $P_\tau$ in $\C$, with vertices $0$, $1$, $\tau$, and $1+\tau$ with $\textrm{Im}(\tau)>0$.
We use the Theta function $\Theta$ and the first integral $Z$ to associate a function
\[
M(p,s)=\frac{\Theta'\left(Z(p)-Z(s)+z_0\right)}{\Theta\left(Z(p)-Z(s)+z_0\right)}\, ,
\]
where $z_0=(1+\tau)/2$ is the unique zero of $\Theta$ in the parallelogram $P_\tau$.
This allows us to introduce a Cauchy-Pompeiu type operator in $\R^2$ given
\[
T_\omega g(p)=\frac{1}{2\pi i}\,\int_R M(p,s)g(s)\, d\mu_s\,.
\]
The properties of this operator are summarized in Theorem \ref{propertiesofT} and used to study
the solvability of $L$ on $\T^2$.

We prove in Theorem \ref{Lu=f} that if $f\in L^q(\T^2)$ with $q>2+\sigma$, where $\sigma$
is a positive number associated to $\omega$, then equation $Lu=f$ has a H\"{o}lder continuous solution in $\T^2$
if and only if $\dis\int_{\T^2}f=0$.
For  $A\in L^q(\T^2)$, we show in Theorem \ref{Lu=Au}, that equation $Lu=Au$ has a solution if and only if
$\dis\frac{1}{2\pi i}\,\int_{\T^2}A$ is in the lattice generated by $1$ and $\tau$ in $\C$.
Finally, in Theorem \ref{Lu=Au+Bbaru} we give a necessary and sufficient condition for the solvability of the equation $Lu=Au+B\ov{u}$ and deduce a similarity principle with  the solutions of $Lu=0$. As a consequence we
 show that any solution on $\T^2$ of $Lu=Au+B\ov{u}$ has the form $u=\exp(s)$ with $s$ continuous on $\T^2$.
 
\bigskip

This paper was written when the second author was visiting the Department of Mathematics and Statistics at Florida International University.
He is grateful and would like to thank the host institution for the support provided during the visit.

\section{ A class of hypocomplex structures}
We define a class of differential forms on the two dimensional torus $\mathbb{T}^2$ and associate a global first integral
on the universal covering space $\mathbb{R}^2$.
Let
\begeq
\omega =a(x,y)dx+b(x,y)dy
\stopeq
be a non-vanishing closed one-form on the two dimensional torus $\T^2$ where $(x,y)$ are the angular coordinates.
We assume that $a$ and $b$ are functions of class $C^{1+\varepsilon}$, $\varepsilon>0$, and that $\omega$ satisfies the following properties:
\begin{itemize}
\item[(a)] The set of non-ellipticity
\[
\Sigma =\left\{p\in\T^2;\ \omega(p)\wedge\ov{\omega(p)} =0 \right\}
\]
is a one-dimensional manifold;
\item[(b)] For each connected component $\Sigma_i$ of $\Sigma$, there exists a positive
number $\sigma_i$ such that for every $p\in \Sigma_i$
\[ 
\omega\wedge\ov{\omega} =\abs{\rho_i}^{\sigma_i}g\, dx\wedge dy
\] 
in a neighborhood of $p$, where $\rho_i$ is  defining function  of $\Sigma_i$ near $p$
and $g$ a non-vanishing function;
\item[(c)] The differential form $\omega$ is hypocomplex (see \cite{BCH} and \cite{Tre}). This is equivalent to
$\omega$ having locally open first integrals. That is, for every $p\in\mathbb{T}^2$,
there exist an open set $U\subset\mathbb{T}^2$ with  $p\in U$ and a function $\zeta\in C^{1+\ep}$
 such that $d\zeta\wedge\omega=0$ and
 $\zeta:\, U\,\longrightarrow\, \zeta(U)\,\subset\, \C$
is a homeomorphism.
\end{itemize}

\begin{remark}\label{abarb}
Assumption (c) implies local solvability of $\omega$ (Condition (P) of Nirenberg-Treves) which in turn implies that the function
$\textrm{Im}(a\ov{b})$ does not change sign (see \cite{BCH} and \cite{Tre}).
\end{remark}

As in {\cite{CDM}}, we can assume that there exist local coordinates $(s,t)$ near points $p \in \Sigma_i$ in which the differential form $\omega$ is a multiple
of
\begeq
ds+i\abs{t}^{\sigma_i}dt
\stopeq
and first integral $\zeta_i$
\[
\zeta_i =s+i\frac{t\abs{t}^{\sigma_i}}{\sigma_i+1}.
\]
We denote by $L$ be the orthogonal vector field of $\omega$:
\begeq
L=b(x,y)\dd{}{x}-a(x,y)\dd{}{y}\,.
\stopeq
Let $\dis \Pi :\, \R^2\,\longrightarrow\, \T^2$ be the covering map and denote by
$R$ the fundamental rectangle:
\begeq\label{R}
R=\left\{ (x,y)\in\R^2 ;\ 0\leqslant x \leqslant1,\ \ 0\leqslant y \leqslant1\, \right\}\,.
\stopeq
We consider the pullback
\[
\Omega =\Pi^\ast\omega =\Pi^\ast a(x,y)dx+\Pi^\ast b(x,y)dy\quad\textrm{and}\quad \vfl=\Pi^\ast b(x,y)\dd{}{x}-\Pi^\ast a(x,y)\dd{}{y}\,.
\]
Hence $\Pi^\ast a$ and $\Pi^\ast b$ are doubly periodic in $\R^2$:
\begeq\begin{array}{lll}
\Pi^\ast a(x+1,y) & =\ \Pi^\ast a(x,y+1)& =\ \Pi^\ast a(x,y)\, ;\\
\Pi^\ast b(x+1,y) & =\ \Pi^\ast b(x,y+1)& =\ \Pi^\ast b(x,y)\,.
\end{array}\stopeq
It should be noted that it follows from $d\omega =0$ that $\omega$ is locally exact and that
the function
\begeq\label{primitive}
Z(x,y)=\int_{(0,0)}^{(x,y)}\Omega
\stopeq
is a global first integral of $\Omega$. Furthermore it follows from the double
periodicity of $\Omega$
that there exist constants $C_1,\, C_2\,\in\, \C$ suct that for every $(x,y)\in\R^2$
\begeq
Z(x+1,y)=Z(x,y)+C_1\quad\textrm{and}\quad Z(x,y+1)=Z(x,y)+C_2\,.
\stopeq

\begin{lemma}
$\mathrm{Im}(C_1\ov{C_2})\ne 0\,.$
\end{lemma}

\begin{proof} Since $\textrm{Im}(a\ov{b})$ does not change sign (see Remark \ref{abarb}), then
\begeq\label{IMCC1}
\int_R\Omega\wedge\ov{\Omega} =\int_R 2i\, \Pi^\ast\textrm{Im}(a\ov{b})\, dxdy\, \ne 0\,.
\stopeq
Let $l_1=[0,\ 1]\times\{0\}$ and $l_2=\{0\}\times [0,\ 1]$ be sides of the rectangle $R$. Then using
properties of $Z$, we can write
\begeq\label{IMCC2}\begin{array}{ll}
\dis \int_R\Omega\wedge\ov{\Omega} & =\dis  \int_RdZ\wedge d\ov{Z} =\int_{\pa R}Zd{\ov{Z}}\\ \\
& =\dis \int_{l_1}(-C_2)d\ov{Z}+\int_{l_2}C_1d\ov{Z}=C_1\ov{C_2}-C_2\ov{C_1}\,.
\end{array}\stopeq
The conclusion follows from (\ref{IMCC1}) and (\ref{IMCC2}). \qed
\end{proof}

After replacing $\omega$ by $\dis\frac{\omega}{C_1}$ and, if necessary, after a change of variables
$\widetilde{x}=x,\ \widetilde{y}=-y$, we can assume that the primitive $Z$ satisfies
\begeq\label{Zquasiperiodic}\begin{array}{l}
 Z(x+1,y)=Z(x,y)+1\, ,\\
Z(x,y+1)=Z(x,y)+\tau\quad\textrm{with}\quad
\textrm{Im}(\tau) >0\,.
\end{array}\stopeq

\begin{proposition}
The primitive $Z:\, \R^2\,\longrightarrow\,\C$ given by \eqref{primitive} is  a global homeomorphism.
\end{proposition}

\begin{proof}
First we show that $Z(\R^2)$ is a closed subset of $\C$. Suppose that
$\{(x_n,y_n)\}_n $ is a sequence in $\R^2$ such that $\{Z(x_n,y_n)\}_n $ converges to a point
$q\in\C$. For every $n$, we can find $\alpha_n,\, \beta_n\, \in \Z$ and $(x_n^0,y_n^0)\,\in R$ such that
\[
(x_n,y_n)=(x_n^0+\alpha_n,y_n^0+\beta_n).
\]
Hence
\begeq\label{convergenceZn}
Z(x_n,y_n)=Z(x_n^0,y_n^0)+\alpha_n+\beta_n\tau \,.
\stopeq
The sequence $\{(x_n^0,y_n^0)\}_n \subset R$ is bounded and so is the sequence $\{Z(x_n^0,y_n^0)\}_n$.
It follows then from the convergence of $Z(x_n,y_n)$, (\ref{convergenceZn}), and $\textrm{Im}(\tau)>0$ that
$\alpha_n$ and $\beta_n$ are bounded sequences in $\Z$. Therefore, the sequence $\{(x_n,y_n)\}_n $ is bounded in
$\R^2$ and consequently  $(x_n,y_n)$ converges to a point $p\in\R^2$ and so $Z(p)=q$.

Since $Z$ is also a local homeomorphism (see assumptions on $\omega$), then $Z(\R^2)$ is also open in $\C$.
Hence $Z(\R^2)=\C$. This means that $Z:\, \R^2\,\longrightarrow\,\C$ is a covering map and, therefore, it is
a homeomorphism since $\C$ is simply connected.
\qed
\end{proof}

\begin{remark}
It follows from the hypotheses on $\omega$ that the vector field $\mathbf{L}$ is
hypocomplex in $\R^2$ (see \cite{BCH}, \cite{Tre}). In particular if a function $U$ solves
$\mathbf{L}U=0$ in a region $S\subset\R^2$, then $U$ can be written as $U=H\circ Z$
with $H$ a holomorphic function in $Z(S)\subset\C$.
\end{remark}

\section{An integral operator via the Theta function}
We use the Theta function to construct a generalized Cauchy-Pompeiu operator for the vector field
$\vfl$ that enables us to construct solutions on the torus. For $\tau\in \C$ with $\textrm{Im}(\tau)>0$, consider the Theta function
\begeq
\Theta (z)=\sum_{m\in\Z}\ei{2\pi i m^2\tau}\ei{2\pi imz}\,.
\stopeq
The following properties of $\Theta$ will be used (for details see \cite{N}).

\begin{itemize}
\item[($\imath$)]  $\Theta (z+1)=\Theta (z)\,$;
\item[($\imath\imath$)] $\Theta(z+\tau)=\ei{-i\pi \tau -2\pi iz}\Theta (z)\,$;
\item [($\imath\imath\imath$)] The only zero of $\Theta(z)$ in the parallelogram $P_\tau$ with vertices
$0$, $1$, $\tau$, and $1+\tau$ is simple and is given by
\[ 
z_0=\frac{1+\tau}{2}\,.
\] 
The zeros of $\Theta$ in $\R^2$ are $z_{jk}=z_0+j+k\tau$ with $j,k\in\Z$.
\end{itemize}

\bigskip

For $p,\, s\,\in\R^2$, define the function $M(p,s)$ by
\begeq\label{M(p,s)}
M(p,s)=\frac{\Theta'\left(Z(s)-Z(p)+z_0\right)}{\Theta\left(Z(s)-Z(p)+z_0\right)}\,.
\stopeq
The function $M$ is meromorphic in $Z$ and satisfies the following

\begin{lemma}
For every $p\in\R^2$ and $s$ near $p$, we have
\begeq\label{Property1ofM}
M(p,s)=\frac{1}{Z(s)-Z(p)} + H(Z(s))
\stopeq
with $H$ a holomorphic function near $Z(p)$. Furthermore,
for each $j,\, k\,\in\Z$
\begeq\label{Property2ofM}
M(p,s+(j,k))=M(p,s)-2\pi i k\,.
\stopeq
\end{lemma}

\begin{proof}
Property (\ref{Property1ofM}) follows directly from the definition (\ref{M(p,s)}) of $M$ and
the properties of the $\Theta$ function.
To verify (\ref{Property2ofM}), notice that since
\[
\Theta (z+j+k\tau)=\ei{-i\pi k\tau -2\pi i kz}\Theta (z)
\]
then
\[
\Theta' (z+j+k\tau)=\ei{-i\pi k\tau -2\pi i kz}\left[ \Theta' (z)-2\pi ik\Theta (z)\right]\,.
\]
Therefore
\[\begin{array}{ll}
M(p,s+(j,k))& =\dis\frac{\Theta'\left(Z(s+(j,k))-Z(p)+z_0\right)}{\Theta\left(Z(s+(j,k))-Z(p)+z_0\right)}\\
& =\dis\frac{\Theta'\left(Z(s)-Z(p)+z_0+j+k\tau\right)}{\Theta\left(Z(s)-Z(p)+z_0+j+k\tau\right)}\\
& = \dis\frac{\Theta'\left(Z(s)-Z(p)+z_0\right)}{\Theta\left(Z(s)-Z(p)+z_0\right)} -2\pi i k = M(p,s)-2\pi ik\,.\qed
\end{array}\]
\end{proof}

Now we use the function $M$ as the kernel of the operator $T_\omega $ defined defined for $g\in L^q(\R^2)$ by
\begeq\label{operatorT}
T_\omega g(p)=\frac{1}{2\pi i}\int_RM(p,s)g(s)d\mu_s
\stopeq
where $d\mu_s$ is the density measure in $\R^2$.
A simple version of this operator was considered in \cite{Mez} and \cite{Mez2} for other classes of vector fields, and more recently in \cite{CDM} and \cite{CDM1}.

Let
\begeq\label{sigma}
\sigma =\max_{1\leqslant i\leqslant N}\sigma_i\, ,
\stopeq
where $\sigma_i$ is the positive number associated with the connected component $\Sigma_i$ of the
characteristic set $\Sigma$ given in hypothesis (b) on $\omega$ 
and where $N$ is the number of connected components of $\Sigma$.

It follows from property (\ref{Property1ofM}) of $M$ and from Theorem 16 in \cite{CDM} that for $g\in L^q(\R^2)$ with $q>2+\sigma$, we have
\begeq \label{alpha}
T_\omega g\in C^\alpha(R),\ \ \textrm{with}\ \
\alpha =\frac{2-p-\mu}{p}\, ,\ p=\frac{q}{q-1}\, , \ \textrm{and}\
\mu=\frac{\sigma}{\sigma+1}\,.
\stopeq

\begin{proposition}\label{testfunction}
Let $\phi\,\in\, C^\infty_0(\R^2)$. Then for every $p\in R$ we have
\begeq\label{valuesofphi}
\sum_{j,k\in\Z}\phi(p+(j,k)) =\frac{-1}{2\pi i}\,\int_{\R^2} M(p,s)\mathbf{L}\phi(s)\, d\mu_s.
\stopeq
\end{proposition}

\begin{proof}
Let $p_\ast $ be a point in the interior of the rectangle $R$, $z_\ast=Z(p_\ast)$ and $D_\ep$ be the disc with
center $z_\ast$ and radius $\ep >0$. We take $\ep$ small enough so that $D_\ep\,\subset\, R$. Set
\[
K_\ep^{jk}=Z^{-1}(D_\ep +j+k\tau)\quad\textrm{and}\quad
\R^2_\ep =\R^2\backslash \bigcup_{j,k\in\Z}K^{jk}_\ep\,.
\]
Using the fact that $\mathbf{L}M(p_\ast,s)=0$ in $\R^2_\ep$ and $\supp (\phi)$ is compact, then
 Green's Theorem applied to the function $M(p_\ast,s)\vfl \phi(s)$ in a domain containing $\supp (\phi)$ gives
\begeq\label{Greenforphi}
\int_{\R^2_\ep} M(p_\ast,s)\vfl \phi(s)\,d\mu_s =-
\sum_{j,k\in\Z}\int_{\pa K^{jk}_\ep} M(p_\ast,s)\phi(s)\, dZ(s)\,.
\stopeq
Properties (\ref{Property1ofM}) and (\ref{Property2ofM}) together with a change of variables in the integrals over
$\pa K^{jk}_\ep$ give
\begin{multline}\label{integralKepsilon}
\dis\int_{\pa K^{jk}_\ep}\!\!\! M(p_\ast,s)\phi(s)\, dZ(s) =\dis\int_{\pa K^{00}_\ep} \!\!\! M(p_\ast,s+(j,k))\phi(s+(j,k))\, dZ(s)\\ \\
=\dis\int_{\pa D_\ep}\!\!\! \left( M(p_\ast,Z^{-1}(\zeta))-2\pi ik\right)\phi\left(Z^{-1}(\zeta)+(j,k)\right)\, d\zeta\\ \\
 =\dis\int_0^{2\pi}\!\!\! \left[ \frac{1}{\ep\ei{i\theta}}+H(\ep\ei{i\theta})-2\pi ik\right]\phi\left( Z^{-1}(z_\ast+\ep\ei{i\theta})+(j,k)\right)
i\ep\ei{i\theta}d\theta\,.
\end{multline}
Formula (\ref{valuesofphi}) follows from (\ref{Greenforphi}) and (\ref{integralKepsilon}) by taking $\ep\rightarrow 0\,.$
\qed
\end{proof}

We have the following theorem:

\begin{theorem}\label{propertiesofT}
For every  function $P\in L^q(\R^2)$ with $q>2+\sigma$,
the function $T_\omega P \in C^\alpha(\R^2)$ \emph(with $\alpha$ given in \eqref{alpha}\emph)
satisfies
\begin{itemize}
\item[($\imath$)] $T_\omega P(x+1,y)=T_\omega P(x,y)\,$;
\item[($\imath\imath$)] $T_\omega P(x,y+1)=T_\omega P(x,y)-\dis\int_RP(s)d\mu_s\,$; and
\item[($\imath\imath\imath$)] If in addition $P$ is doubly periodic, then  $\vfl T_\omega P=P\,.$
\end{itemize}
\end{theorem}

\begin{proof}
Properties ($\imath$) and ($\imath\imath$) follow directly from (\ref{Property2ofM}). To verify  ($\imath\imath\imath$),
let $\phi\in C^\infty_0(\R^2)$. Then using Proposition \ref{testfunction} we find
\begeq\begin{array}{ll}
<\vfl T_\omega P,\, \phi> & =\dis -<T_\omega P,\, \vfl \phi>=-\int_{\R^2}T_\omega P(p)\vfl \phi(p)\, d\mu_p\\ \\
& =\dis\int_R\frac{-1}{2\pi i}\left[\int_{\R^2}M(p,s)\vfl\phi(p)d\mu_p\right]\,P(s)\, d\mu_s\\ \\
& =\dis\int_R\sum_{j,k\in\Z}\phi(s+(j,k))\, P(s)\, d\mu_s\\ \\
 &=\dis\int_R\sum_{j,k\in\Z}\phi(s+(j,k))\, P(s+(j,k))\, d\mu_s\\ \\
& =\dis\int_{\R^2}P(s)\phi(s)\, d\mu_s = <P,\phi>\,. \qed
\end{array}\stopeq
\end{proof}

\begin{theorem}\label{Lu=f}
For $f\in L^q(\T^2)$ with $q>2+\sigma$, equation $Lu=f$ has a solution
$u\in C^\alpha(\T^2)$ if and only if $\dis\int_{\T^2}f\, dxdy =0$.
\end{theorem}

\begin{proof}
If equation $Lu=f$ is solvable on $\T^2$,
\[
\int_{\T^2}f\, dxdy=\int_{\T^2}Lu\, dxdy=\int_{\T^2}d(udZ)=0.
\]
Conversely if $\dis\int_{\T^2}f\, dxdy=0$, then it follows from Theorem \ref{propertiesofT} that $T_\omega \Pi^\ast f$ is doubly
periodic and descends as a solution of $Lu=f$ on $\T^2$. \qed
\end{proof}

\section{The equation $Lu=Au$ on $\T^2$}\label{sec4}
We give a necessary and sufficient condition for the global solvability of the equation $Lu=Au.$ For $A\in L^q(\T^2)$ with $q>2+\si$, we associate the number
\begeq\label{nuofa}
\nu(A)=\frac{-1}{2\pi i}\int_{\T^2}A(p)\, d\mu_p =\frac{T_\omega \Pi^\ast A(0,1)-T_\omega\Pi^\ast A(0,0)}{2\pi i}\,.
\stopeq

\begin{theorem}\label{Lu=Au}
For a function $A\in L^q(\T^2)$ with $q>2+\si$, equation
\begeq\label{equationLu=Au}
Lu=Au
\stopeq
has a solution in $C^\alpha(\T^2)$ if and only if the associated number given by \eqref{nuofa} is in the
lattice generated by 1 and $\tau$:
\begeq\label{nuinlattice}
\nu(A)=j+k\tau\quad\textrm{with}\quad j,\, k\, \in\Z\,.
\stopeq
In this case any solution of \eqref{equationLu=Au} has the form
\begeq\label{solutionsofLu=Au}
u(p)=C\exp(T_\omega\Pi^\ast A (p)+kZ(p))\quad\textrm{with}\quad C\in\C\,.
\stopeq
\end{theorem}

\begin{proof}
Suppose that $\nu(A)$ is given by (\ref{nuinlattice}). The function $V\in C^\alpha(\R^2)$ given by
\[
V(x,y)=T_\omega \Pi^\ast A(x,y)-2\pi i kZ(x,y)
\]
satisfies $\vfl V=\Pi^\ast A$ by Theorem \ref{propertiesofT}, and by (\ref{nuinlattice}) it satisfies
\[
V(x+1,y)=V(x,y)-2\pi i k\ \ \textrm{and}\ \ V(x,y+1)=V(x,y)+2\pi i j\,.
\]
Hence $u(x,y)=\ei{V(x,y)}$ is doubly periodic and satisfies \eqref{equationLu=Au}.

To prove the necessity of (\ref{nuinlattice}), suppose equation (\ref{equationLu=Au}) has a solution
in $\T^2$ (note that in this case the solution is necessarily H\"{o}lder continuous by results contained in
\cite{CDM}). Then the function
\[
V(x,y)=\Pi^\ast u(x,y)\ei{-T_\omega\Pi^\ast A(x,y)}
\]
satisfies $\mathbf{L}V=0$ in $\R^2$. Hence there exists an entire function $H$ such that
$V=H\circ Z$. Furthermore, if $z=Z(x,y)$, then it follows from Theorem \ref{propertiesofT} that
\begeq\label{propofH}\left\{\begin{array}{ll}
H(z+1) &= H(z)\\
H(z+\tau) &=\dis \Pi^\ast u(x,y)\ei{-T_\omega\Pi^\ast A(x,y) -2i\pi \nu(A)}=\ei{ -2i\pi \nu(A)}H(z)\,.
\end{array}\right.\stopeq
It follows from (\ref{propofH}) that $H$ can factored through a function defined on the cylinder. That is,
H can be written as
\[
H(z)=K(\zeta)\quad\textrm{with}\quad \zeta=\ei{2\pi iz}\,
\]
where $K$ is a holomorphic function in the punctured plane $\C^\ast$.
Moreover, $K$ satisfies
\begeq\label{propofK}
K(\zeta\ei{2\pi i\tau})=H(z+\tau)=\ei{ -2i\pi \nu(A)}K(\zeta)\,.
\stopeq
Consider the Laurent series of $K$: $K(\zeta)=\dis\sum_{m\in\Z}a_m\zeta^m$. It follows at once from
(\ref{propofK}) that
\begeq\label{Laurentcoeff}
a_m\ei{2i\pi m\tau}=a_m\ei{ -2i\pi \nu(A)}\,,\qquad\forall m\in\Z\,.
\stopeq
Recall that $\textrm{Im}(\tau)>0$ so that $\ei{2i\pi m\tau}\ne 1$ for all $m$. Hence, system (\ref{Laurentcoeff}) has
a solution if and only if $\nu(A)=j+k\tau$ for some $j,\, k\in\Z$ and in this case
$K(\zeta )=a_k\zeta^k$. The function $\Pi^\ast u$ is therefore
\[
\Pi^\ast u(x,y)=a_k\ei{T_\omega\Pi^\ast A(x,y)+2\pi ikZ(x,y)}\,.\quad \qed
\]
\end{proof}

\section{The equation $Lu=Au+B\overline u$ on $\mathbb T^2$}\label{sec5}

In this section we give a necessary and sufficient condition for the solvability of the equation
\begin{equation}\label{Lu=Au+Bbaru}
Lu=Au+B\overline u
\end{equation}
 on $\mathbb T^2$ and
deduce a similarity principle with the solutions of $Lu=0$ on $\mathbb T^2$  (which are in fact constant functions).
Let $A,\, B\in L^q(\mathbb T^2)$, with $q>2+\sigma$ where $\sigma$ is given in (\ref{sigma}).
For $k\in\mathbb{Z}$, define the operator $P_k$ by
\[
P_kv(x,y)=T_{\omega}\left[\Pi^\ast A+\tilde{B}_k\,\cdot\,\frac{\overline{e^v}}{e^v}\right](x,y)\,,
\]
where
\[
\tilde{B}_k (x,y) = \Pi^\ast B (x,y)\, \exp \left[ -2\pi i k (\overline{Z}(x,y)+Z(x,y))\right]\, .
\]
It follows from \cite{CDM} and property \eqref{Property1ofM} of $M$ that if $v\in L^q(\mathbb{R}^2)$ with
$q>2+\sigma$, then $P_kv\in C^\alpha (R)$ with $\alpha$ given in \eqref{alpha}. We restrict the action of $P_k$ to the subspace
$C(R)$.

\begin{proposition}\label{fixedpoint}
The operator $P_k$ has a fixed point in $C(R)$.
\end{proposition}

\begin{proof} It follows from  \cite[Theorem 9]{CDM}  that there exists $M>0$ such that
\[
|T_{\omega}F(x,y)|\,\leqslant\,  M\|F\|_q\,,\quad \forall F\in L^q(R) \ \mathrm{and} \ (x,y)\in R\, .
\]
Hence
\[
\| P_k v\|_\infty \,\leqslant\, M\, (\|A\|_q+\|B\|_q)\, \doteq C\, \quad \forall v\in V.
\]
Consider the subset $K$ given by
\[
K=\{\, v\in C(R);\ \|v\|_\infty \leqslant C\,\}\, .
\]
$K$ is a compact and convex subset in $C(R)$. For every $v\in K$ we have
\[
\|P_k(v)\|_\infty \,\leqslant\, M\, (\|\Pi^\ast A\|_q+\|\tilde{B}_k\, \exp(\bar{v}-v)\|_q)\,
\leqslant\, M\, (\|A\|_q+\|B\|_q)\, \doteq C\,.
\]
Hence $P_k(K)\subset K$. Furthermore $P_k:\, C(R)\, \longrightarrow\, C(R)$ is continuous.
Indeed, since the function $g(\zeta)=\exp(\ov{\zeta}-\zeta)$ is Lipschitz (with constant 2) in $\C$,
then for $v,\, v_0\, \in C(R)$, we have
\[\begin{array}{ll}
\|P_k(v)-P_k(v_0)\|_\infty & =
\,  \left\lVert T_{\omega}\left[\tilde{B}_k\cdot\left(\exp (\bar{v}-v)-\exp (\bar{v}_0-v_0) \right)\right]\right\rVert_{\infty}\\ \\
 & \leqslant\, M \|\tilde{B}_k\|_q\, \left\lVert \exp (\bar{v}-v)-\exp (\bar{v}_0-v_0) \right\rVert_{\infty}\\ \\
  & \leqslant\, 2 M \|\tilde{B}_k\|_q\, \|v-v_0 \|_\infty\, .
\end{array}\]
Thus $P_k$ has a fixed point in $K$ (Schauder's Fixed Point Theorem).  \qed
\end{proof}
\bigskip

Note as in that in Theorem \ref{propertiesofT}, for all $x,\, y\, \in [0, 1]$, $P_k$ satisfies
\begin{equation}\label{PropertiesofPk}
P_kv(1,y)=P_kv(0,y) \ \mathrm{and}\
P_kv(x,1)=P_kv(x,0) -\!\!\int_R\!\!\! \,\left(\Pi^\ast A+\tilde{B}_k\cdot\frac{\overline{e^v}}{e^v}\right)\, d\mu\,.
\end{equation}
Let $V_k$ be the set of fixed points of $P_k$: $\displaystyle V_k=\{\, v\in C(R),\ P_kv=v\, \}$.
Hence, for every $v\in V_k$, there is $\nu\in \C$ such that
 $v(1,y)=v(0,y)$ and $v(x,1)-v(x, 0)=\nu$. Let
\[ \Lambda_k\doteq\{\, v(x,1)-v(x,0): v\in V_k\}\,.\]

\medskip

\begin{theorem}\label{SolvabilityofLu=Au+Bbaru}
Equation \eqref{Lu=Au+Bbaru} has a H{\"o}lder continuous solution  on $\mathbb T^2$ if and only if
there are $j,k\in \mathbb Z$ such that $2\pi i(j-k\tau)\in \Lambda_k.$
Moreover any solution $u$ is such that
\[
 \Pi^\ast u(x,y)=C\exp \left(2\pi ik Z(x,y) + P_k v(x,y)\right )\, ,
\]
with $C\in\mathbb C$, and $v \in V_k$.
\end{theorem}

\medskip

\begin{proof} Suppose that there is $v\in V_k$ with $v(x,1)-v(x,0)=2\pi i(j-k\tau)$, for some $j,k\in\mathbb Z$. Let
\[
U(x,y) = \exp\left[2\pi i k Z(x,y)+ P_k v(x,y)\right]\,,\quad  (x,y)\in \mathbb R^2\, .
\]
It follows from property (\ref{PropertiesofPk}) and assumption on $v$ that $U$ is doubly periodic, and, as $P_kv=v$,
we have
\[\begin{array}{ll}
\vfl(U) & =\dis  U\vfl(P_kv)=U\vfl\left(T_{\omega}\left[\Pi^\ast A+ \tilde{B}_k\cdot\frac{\overline{e^v}}{e^v}\right]\right)\\ \\
& =\dis U\left(\Pi^\ast A+\Pi^\ast B\cdot\frac{\overline U}{U} \right)=\Pi^\ast A \cdot U+\Pi^\ast B \cdot {\overline U}\,.
\end{array}\]
Since $U$ is doubly periodic then $u=U\circ\Pi^{-1}\,\in C^\alpha(\mathbb{T}^2)$ satisfies $Lu=Au+\overline{u}$.

Conversely, suppose that $u\in C^\alpha(\mathbb{T}^2)$ solves \eqref{Lu=Au+Bbaru}. Since $L$ is elliptic on
$\mathbb{T}^2\backslash\Sigma$, it follows that the zeros of $u$ are isolated on $\mathbb{T}^2\backslash\Sigma$ and
the function $\dis\frac{\overline{u}}{u}\, \in L^\infty (\mathbb{T}^2)$.
Let
\[
V(x,y)=T_\omega \left[ \Pi^\ast A +\Pi^\ast B\, \Pi^\ast \left(\frac{\overline{u}}{u}\right)\right]\, .
\]
Theorem \ref{propertiesofT} implies that
\[
\vfl (V)= \Pi^\ast A +\Pi^\ast B\, \Pi^\ast\left(\frac{\overline{u}}{u}\right)
\]
and there exists $\beta\in\C$ such that
\begin{equation}\label{V-function}
V(x,y)=V(x+1,y)\quad\mathrm{and}\quad V(x,y)=V(x,y+1)+\beta,\quad \forall (x,y)\in\R^2\, .
\end{equation}
We have
\[
\vfl (\Pi^\ast u \, \ei{-V})=0\,
\]
in $\R^2$. Therefore, there exists an entire function $H$ in $\C$ such that
\[
\Pi^\ast u(x,y) \ei{-V(x,y)}=H(Z(x,y))\qquad\forall (x,y)\in\R^2\, .
\]
Moreover the double periodicity of $\Pi^\ast u$ and property \eqref{V-function} imply that the
entire function $H$ satisfies
\[
H(z+1)=H(z)\ \mathrm{ and }\ H(z+\tau)=e^{\beta}H(z)\,,\quad \forall z\in \C\, .
\]
As in the previous section, such an entire function $H$ is of the form $H(z)=Ce^{2\pi ikz}$ with $C\in \C$ and
$\beta=2\pi i(j+k\tau)$ for some $j,k\, \in\Z$. This completes the proof. \qed
\end{proof}
\begin{remark}
In particular, we have showed that a solution to $Lu=Au+B\overline u$ globally defined on $\mathbb T^2$ never vanishes if it is not identically zero.
\end{remark}


\end{document}